%% file: paper_arxiv.tex
\documentclass[11pt]{article}
\usepackage{amsmath}
\usepackage{textcomp}
\usepackage{amsbsy}
\usepackage{latexsym}
\usepackage{amsfonts}
\usepackage{amsthm}
\usepackage{amssymb}
\usepackage{algorithmic}
\usepackage{algorithm}
\usepackage{subfigure}
\usepackage[usenames,dvipsnames]{xcolor}
\usepackage{tikz,ifthen,fp}
\usepackage{pgfplots}
\usetikzlibrary{backgrounds,arrows}

\input{macros}

\usepackage{hyperref}
\usepackage[makeroom]{cancel}

\usepackage{fullpage}

\newcommand{\red}[1]{{\color{black}{#1}}}
\newcommand{\blue}[1]{{\color{black}{#1}}}

\newtheorem{lemma}{Lemma}[section]
\newtheorem{proposition}[lemma]{Proposition}
\newtheorem{corollary}[lemma]{Corollary}
\newtheorem{theorem}[lemma]{Theorem}
\newtheorem{remark}[lemma]{Remark}

\newtheorem{algo}[lemma]{Algorithm}

\begin{document}

\bibliographystyle{plain}

\title{Stochastic methods for solving high-dimensional partial differential equations}

\author{Marie Billaud-Friess\footnotemark[1] \and Arthur Macherey\footnotemark[1] \footnotemark[3] \and Anthony Nouy\footnotemark[1] \and Cl\'ementine Prieur\footnotemark[2]}

\renewcommand{\thefootnote}{\fnsymbol{footnote}}
\footnotetext[1]{Centrale  Nantes, LMJL, UMR CNRS 6629, 1 rue de la No\"e, 44321 Nantes}
\footnotetext[2]{Univ. Grenoble Alpes, Inria, CNRS, Grenoble INP*, LJK, 38000 Grenoble, France}
\renewcommand{\thefootnote}{\fnsymbol{footnote}}
\footnotetext[3]{Corresponding author (arthur.macherey@ec-nantes.fr).}

\maketitle

\begin{abstract} We propose algorithms for solving high-dimensional Partial Differential Equations (PDEs) that combine a probabilistic interpretation of PDEs, through Feynman-Kac representation, with sparse interpolation. Monte-Carlo methods and time-integration schemes are used to estimate pointwise evaluations of the solution of a PDE. We use a sequential control variates algorithm, where control variates are constructed based on successive approximations of the solution of the PDE. Two different algorithms are proposed, combining in different ways the sequential control variates algorithm and adaptive sparse interpolation. Numerical examples will illustrate the behavior of these algorithms.
\end{abstract}




\section{Introduction}
\label{sec1}
We consider the solution of an elliptic partial differential equation
\begin{equation}
\begin{split}
\calA(u) & = g \quad \text{in} \quad \calD, \\
u & = f \quad \text{on} \quad \partial \calD,
\end{split}
\label{eq:ellipticPDE}
\end{equation}
where $u : \overline{\calD} \rightarrow \bbR$ is a 
real-valued function, and $\calD$ is an open bounded domain in $\bbR^d$. $\calA$ is an elliptic linear differential operator and $f : \partial \calD \to \bbR$, $g: \overline{\calD} \rightarrow \bbR$ are respectively the boundary condition and the source term of the PDE. \\

We are interested in approximating the solution of  \eqref{eq:ellipticPDE}  up to a given precision.
For high dimensional PDEs  ($d \gg 1$), this requires \blue{suitable approximation formats} such as sparse tensors \cite{BG04,SY10} or low-rank tensors 
\cite{Osel11,GKT13,Hack14,BSU16,Nou17}. 
Also, this requires algorithms that provide approximations in a given approximation format. Approximations are typically provided by Galerkin projections using variational formulations of PDEs. Another path consists in 
 using a probabilistic representation of the solution $u$ through Feynman-Kac formula, and Monte-Carlo methods to provide 
estimations of pointwise evaluations of $u$  (see e.g., \cite{G2016}). This allows to compute approximations in a given approximation format through classical interpolation or regression \cite{Beck17,WJJ17,Beck18}. In \cite{GM2004,GM2005}, the authors consider interpolations on fixed polynomial spaces and propose a sequential control variates method for improving the performance of Monte-Carlo estimation. 
In this paper, we propose algorithms that combine this variance reduction method 
with adaptive sparse interpolation \cite{Chkifa2013,Chkifa2014}.

The outline is as follows. In section \ref{sec2}, we recall the theoretical and numerical aspects associated to probabilistic tools for estimating the solution of \eqref{eq:ellipticPDE}. We also present the sequential control variates algorithm introduced in \cite{GM2004,GM2005}. In section \ref{sec3} we introduce sparse polynomial interpolation methods and present a classical adaptive algorithm. In section \ref{sec:4}, we present two algorithms combining the sequential control variates algorithm from section \ref{sec2} and adaptive sparse polynomial interpolation. Finally, numerical results are presented in section  \ref{sec:4}.

\section{Probabilistic tools for solving PDEs}
\label{sec2} 

We consider the problem \eqref{eq:ellipticPDE} with a linear partial differential operator defined by $\calA(u) = -\calL(u) + k u$, where 
 $k$ is a real valued function defined on $\overline \calD$, and where
\begin{equation}
\calL (u)(\bsx) = \dfrac{1}{2} \sum_{i,j=1}^d (\sigma(\bsx) \sigma(\bsx)^T)_{ij} \partial^2_{\bsx_i \bsx_j}u(\bsx) + \sum_{i=1}^d b_i(\bsx) \partial_{\bsx_i}u(\bsx) 
 \label{differentialoperator}
\end{equation}
 is the {\em infinitesimal generator}  associated to the $d$-dimensional diffusion process  $\bsX^{\bsx}$ solution of \blue{the stochastic differential equation}
\begin{equation}
d \bsX^{\bsx}_t = b(\bsX^{\bsx}_t) dt + \sigma(\bsX^{\bsx}_t) d W_t, \quad \bsX^{\bsx}_0 = \bsx \in \overline{\calD},
\label{eq:SDE}
\end{equation}
where $W$ is a $d$-dimensional Brownian motion and \blue{$b := (b_1, \ldots, b_d)^T : \RR^d  \to \RR^d$} and $\sigma : \RR^d \to \RR^{d\times d}$ stand for the drift and the diffusion respectively. 

\subsection{Pointwise evaluations of the solution}

The following theorem recalls the Feynman-Kac formula (see \cite[Theorem 2.4]{CM2015} \blue{or \cite[Theorem 2.4]{Friedman1975} and the references therein}) that provides a probabilistic representation of  $u(\bsx)$, the solution of \eqref{eq:ellipticPDE} evaluated at $\bsx \in \overline{\calD}$.  
\begin{theorem}[Feynman-Kac formula] 
Assume that
\begin{enumerate}
\item[$(H {1})$]  $\calD$ is an open connected bounded domain of $\bbR^d$, regular in the sense that, if  $\tau^{\bsx} = \inf \left\{ s > 0 ~ : ~ \bsX^{\bsx}_s \notin \calD \right\}$ is the first exit time of $\calD$ for the process $\bsX^{\bsx}$, we have
$$
\bbP(\tau^\bsx = 0) = 1, \quad \bsx \in \partial \calD,
$$
\vspace{-0.5cm}
\item[$(H {2})$] $b,\sigma$ are Lipschitz functions,
\item[$(H {3})$] $f$ is continuous on $\partial \calD$, $g$ and $k \geq 0$ are H\"older-continuous functions on $\overline{\calD}$, 
\item[$(H {4})$] (uniform ellipticity assumption) there exists $c > 0$ such that
$$
\sum_{i,j=1}^d \left(\sigma(\bsx) \sigma(\bsx)^T \right)_{ij} \xi_i \xi_j \geq c \sum_{i=1}^d \xi_i^2, \quad \xi \in \bbR^d, \; \bsx \in \overline{\calD}.
$$
\end{enumerate}
Then, there exists a unique solution of \eqref{eq:ellipticPDE} in $\calC \left( \overline{\calD} \right) \cap \calC^2 \left( \calD \right)$, which satisfies for all $\bsx \in \overline{\calD}$  
\begin{equation}
\begin{split}
u(\bsx) & = \mathbb{E} \left[ F(u,\bsX^{\bsx}) \right] \\
\end{split}
\label{eq:FK}
\end{equation} 
where 
\begin{equation*}
F(u,\bsX^{\blue{x}}) =  u(\bsX^{\bsx}_{\tau^{\bsx}}) \exp \left( - \int_0^{\tau^{\bsx}} k(\bsX^{\bsx}_t) dt \right) + \int_0^{\tau^{\bsx}}  \calA (u)(\bsX^{\bsx}_t) \exp \left( - \int_0^t k(\bsX^{\bsx}_s) ds \right) dt,
\end{equation*}
with $u(\bsX^{\bsx}_{\tau^{\bsx}}) = f(\bsX^{\bsx}_{\tau^{\bsx}})$ and $ \calA (u)(\bsX^{\bsx}_t) = g(\bsX^{\bsx}_t)$. 
\label{th:FK}
\end{theorem}

Note that $F(u,X^x)$ in \eqref{eq:FK} only depends on the values of $u$ on $\partial D$ and $\calA(u)$ on $D$, which are the given data \blue{$f$ and $g$} respectively. 
A Monte-Carlo \blue{method} can then be used to estimate 
$u(x)$ using \eqref{eq:FK}, which relies on the simulation of independent samples of an approximation of the stochastic process $\bsX^\bsx$. This process is here approximated 
 by an Euler-Maruyama scheme. More precisely, letting \blue{$t_n = n\Delta t$}, $n\in \mathbb{N}$, $\bsX^{\bsx}$ is approximated by a piecewise constant process $\bsX^{\bsx, \Delta t}$, where $\bsX^{\bsx, \Delta t}_t = \bsX^{\bsx, \Delta t}_n$ for $ t \in [t_n,t_{n+1}[$ and  
\begin{equation}
\begin{split}
\bsX^{\bsx, \Delta t}_{n+1} & = \bsX^{\bsx, \Delta t}_n + \Delta t ~ b(\bsX^{\bsx, \Delta t}_n) + \sigma(\bsX^{\bsx, \Delta t}_n) ~ \Delta W_n, \\
\bsX^{\bsx, \Delta t}_0 & = \bsx.
\end{split}
\label{eq:Eulerscheme}
\end{equation}
Here $\Delta W_n = W_{n+1} -  W_n$ is an increment of the standard Brownian motion. For details on time-integration schemes, the reader can refer to \cite{KP2012}. Letting $\{ \bsX^{\bsx, \Delta t}(\omega_m) \}_{m=1}^M$ be independent samples of $\bsX^{\bsx, \Delta t}$, we obtain an estimation $u_{\Delta t,M}(\bsx)$ of $u(\bsx)$ defined as
\begin{equation}
\begin{split}
 u_{\Delta t,M}(\bsx) & := \dfrac{1}{M} \sum_{m=1}^M F \left( u,\bsX^{\bsx, \Delta t}(\omega_m) \right) \\
& = \dfrac{1}{M} \sum_{m=1}^M \bigg[ f(\bsX^{\bsx,\Delta t}_{\tau^{\bsx,\Delta t}}(\omega_m)) \exp \left( - \int_0^{\tau^{\bsx,\Delta t}} k(\bsX^{\bsx,\Delta t}_t(\omega_m)) dt \right) \\
 & + \int_0^{\tau^{\bsx,\Delta t}}  g(\bsX^{\bsx,\Delta t}_t(\omega_m)) \exp \left( - \int_0^t k({\bsX}^{\bsx,\Delta t}_s(\omega_m)) ds \right) dt \bigg]
\end{split}
\label{eq:approxFK}
\end{equation}
where $\tau^{\bsx,\Delta t}$ is the first exit time of $D$ for the process $\bsX^{\bsx,\Delta t}(\omega_m)$, given by 
$$
\tau^{\bsx,\Delta t}  = \inf \left\{ t > 0 ~ : ~ \bsX^{\bsx,\Delta t}_t \notin \calD \right\}  = \min \left\{ t_n > 0 ~ : ~ \bsX^{\bsx,\Delta t}_{t_n} \notin \calD \right\}.
$$
\begin{remark}
In practice, $f$ has to be defined over $\mathbb{R}^d$ and not only on the boundary $\partial \cal D$. Indeed, although $X^{\bsx}_{\tau^{\bsx}} \in \partial \cal D$ with probability one, $X^{\bsx,\Delta t}_{\tau^{\bsx,\Delta t}} \in \mathbb{R}^d \setminus \overline{\calD}$ with probability one.
 \end{remark}
The error can be decomposed in two terms \\
\begin{equation}
\begin{split}
u(\bsx) -  u_{\Delta t,M}(\bsx) = & \overbrace{u(\bsx) - \bbE \left[ F \left( u,\bsX^{\bsx,\Delta t} \right) \right]}^{\varepsilon_{\Delta t}} \\ 
& + \underbrace{\bbE \left[ F \left( u,\bsX^{\bsx,\Delta t} \right) \right] - \dfrac{1}{M} \sum_{m=1}^M F \left( u,X^{\bsx, \Delta t}(\omega_m) \right)}_{\varepsilon_{MC}},
\end{split}
\label{eq:error}
\end{equation}
where $\varepsilon_{\Delta t}$ is the time integration error and $\varepsilon_{MC}$ is the Monte-Carlo estimation error.
Before discussing the contribution of each of both terms to the error, let us introduce the following additional assumption, which 
ensures that $\calD$ does not have singular points\footnote{Note that together with {$(H {4})$} , assumption {$(H {5})$}  implies {\it $(H {1})$}  (see \cite[\S4.1]{GM2005} for details), so that the set of hypotheses $(H{1})$-$(H{5})$ could be reduced to $(H{2})$-$(H{5})$.}. \\

\noindent {\it $(H {5})$}  Each point of $\partial \calD$ satisfies the {\em exterior cone condition} which means that, for all $\bsx \in \partial \calD$, there exists a finite right circular cone $K$, with vertex $\bsx$, such that $\overline K \cap \overline \calD = \{ \bsx \}$. \\

Under assumptions $(H{1})$-$(H{5})$, it can be proven \cite[\S4.1]{GM2005} that the time integration error $\varepsilon_{\Delta t}$  converges to zero.  It can be improved to $O(\Delta t^{1/2})$ by adding differentiability assumptions on the boundary \cite{GMe2009}. The estimation error $\varepsilon_{MC}$  is a random variable with zero mean and standard deviation converging as $O ( M^{-1/2})$.  \pagebreak The computational complexity for computing \blue{a pointwise evaluation of} $u_{\Delta t,M}(\bsx)$ is in $O \left( M \Delta t^{-1} \right)$ \blue{in expectation for $\Delta t$ sufficiently small}\footnote{\blue{A realization of $X^{x,\Delta t}$ over the time interval $[0,\tau^{x,\Delta t}]$ can be computed in $O \left( \tau^{x,\Delta t}\Delta t^{-1} \right)$. Then, the complexity to evaluate $u_{\Delta t,M}(x)$ is in $O(\bbE(\tau^{x,\Delta t}) M \Delta t^{-1} )$ in expectation. Under  $(H{1})$-$(H{5})$, it is stated in the proof of \cite[Theorem 4.2]{GM2005} that  $\sup_{x} \bbE [\tau^{x, \Delta t}] \le C$ with $C$ independent of $\Delta t$ for $\Delta t$ sufficiently small.}}, so that the computational complexity for achieving a precision $\epsilon$ (root mean squared error) behaves as $O(\epsilon^{-4})$. This does not allow to obtain a very high accuracy in a reasonable computational  time. 
The convergence with $\Delta t$ can be improved to $O(\Delta t)$ by suitable boundary corrections \cite{GMe2009}, therefore yielding a convergence in $O(\epsilon^{-3}).$ To further improve the convergence, high-order integration schemes could be considered (see \cite{KP2012} for a survey). 
Also, variance reduction methods can be used to further improve the convergence, such as antithetic variables, importance sampling, control variates (see  \cite{G2016}). Multilevel Monte-Carlo \cite{Giles2012} can be considered as a variance reduction method using several control variates (associated with processes $X^{\bsx,\Delta t_k}$ using different time discretizations). Here, 
 we rely on the sequential control variates algorithm proposed in \cite{GM2004} and analyzed in \cite{GM2005}. This algorithm constructs a sequence of approximations of $u$. At each iteration of the algorithm, the current approximation is used as a control variate for the estimation of $u$ through 
 Feynman-Kac formula. 

\subsection{A sequential control variates algorithm}
\label{sec2.3}
Here we recall the sequential control variates algorithm introduced in \cite{GM2004} in a general interpolation framework. 
We let $V_\Lambda \subset \calC^2(\overline \calD)$ be an approximation space of finite dimension $\# \Lambda$ and 
 let $\calI_\Lambda : \mathbb{R}^\calD \to V_\Lambda$ be the interpolation operator associated with a unisolvent grid $\Gamma_\Lambda = \{ \bsx_\bsnu : \bsnu \in \Lambda \}$. We let $(\bsl_\bsnu)_{\bsnu \in \Lambda}$ denote the (unique) basis of $V_\Lambda$ that satisfies the interpolation property $\bsl_\bsnu(\bsx_\bsmu) = \delta_{\bsnu\bsmu}$ for all $\bsnu,\bsmu \in \Lambda$. The interpolation $\calI_\Lambda(w) = \sum_{\nu \in \Lambda} w(x_\nu) l_\nu(x)$ of  function $w$ is then the unique function in $V_\Lambda$ such that 
 $$
 \calI_\Lambda(w) (x_\nu) = w(x_\nu) , \quad \nu \in \Lambda.
 $$
 The following algorithm provides a sequence of approximations $(\tilde{u}^k)_{k \geq 1}$ of $u$ in $V_\Lambda$, which are defined 
by $ \tilde{u}^{k}=  \tilde{u}^{k-1} + \tilde{e}^{k}$, where $\tilde{e}^{k}$ is an approximation of $e^k$, solution of  
 \begin{align*}
 \begin{array}{rcll}
&\calA (e^k)(\bsx) =  g(\bsx)  - \calA  (\tilde u^{k-1})(\bsx), & \bsx \in \calD,\\  
&e^k(\bsx) = f(\bsx)- \tilde u^{k-1}(\bsx), &\bsx \in \partial \calD.
 \end{array}
\end{align*}
Note that $e^k$ admits a Feyman-Kac representation $e^k(x) = \mathbb{E}(F(e^k , X^{\bsx}))$, where $F(e^k , X^{\bsx})$ depends on the residuals 
$ g - \calA  (\tilde u^{k-1})$ on $\cal D$ and $f - \tilde u^{k-1}$ on $\partial \calD$. The approximation $\tilde e^k$ is then defined as the interpolation $\calI_\Lambda(e^k_{\Delta t,M})$ of the Monte-Carlo estimate $e^k_{\Delta t,M}(x)$ of $e^k_{\Delta t}(x) =  \mathbb{E}(F(e^k , X^{\bsx,\Delta t}))$ (using $M$ samples of $X^{x,\Delta t}$).


\begin{algo}\normalfont\label{algo:GM}{\bfseries (Sequential control variates algorithm)}
\begin{algorithmic}[1]\normalfont
\STATE Set $\tilde u^{0} = 0$, $k=1$ and $S = 0$. 
\WHILE{$k \leq K ~ \text{and} ~ S < n_s $} 
\STATE Compute $ e^k_{\Delta t, M}(\bsx_\bsnu)$ for $\bsx_\bsnu \in \Gamma_\Lambda$.
\STATE Compute $\tilde e^k = \calI_\Lambda(e^k_{\Delta t,M}) = \sum_{\bsnu \in \Lambda} e^k_{\Delta t,M}(\bsx_\bsnu) \bsl_\bsnu(\bsx) $. \label{seqvarcomputeerror}
\STATE Update $ \tilde{u}^{k}=  \tilde{u}^{k-1} + \tilde{e}^{k}$. 
\STATE If $ \| \tilde u^k - \tilde u^{k-1} \|_2 \leq \epsilon_{tol} \| \tilde u^{k-1} \|_2$  then $S = S +1$ else $S = 0$.
\STATE  Set $k = k+1$. 
\ENDWHILE
\end{algorithmic}
\end{algo}
For practical reasons, Algorithm \ref{algo:GM} is stopped using an heuristic error criterion based on stagnation. This criterion is satisfied when the desired tolerance $\epsilon_{tol}$ is reached for $n_s$ successive iterations (in practice we chose $n_s =5$). \\

Now let us provide some convergence results for Algorithm \ref{algo:GM}. To that goal, we introduce the time integration error at point $\bsx$ for a function $h$ 
\begin{equation}
e^{\Delta t}(h,\bsx) = \mathbb{E} [ F ( h, \bsX^{\Delta t, \bsx}  ) ] - \mathbb{E} [ F \left( h, \bsX^\bsx  \right) ].
\label{eq:timeintegrationerror}
\end{equation}
 Then the following theorem \cite[Theorem 3.1]{GM2005} gives a control of the error in expectation.
\begin{theorem}
Assuming {\it \blue{$(H {2})$}}-{\it $(H {5})$}, it holds
\begin{equation*}
\sup_{\bsnu \in \Lambda} \left| \mathbb{E} \left[ \tilde{u}^{n+1}(\bsx_\bsnu) - u( \bsx_\bsnu) \right] \right| \leqslant C(\Delta t,\Lambda)  \sup_{\bsnu \in \Lambda} \left| \mathbb{E} \left[ \tilde{u}^{n}(\bsx_\bsnu) - u(\bsx_\bsnu) \right] \right| + C_1(\Delta t,\Lambda)
\end{equation*}
with $ \displaystyle C(\Delta t,\Lambda) = \underset{\bsnu \in \Lambda}{\sup} \sum_{\bsmu \in \Lambda} \vert e^{\Delta t}(\bsl_\bsmu,\bsx_\bsnu) \vert$ and $C_1(\Delta t,\Lambda) = \sup_{\bsnu \in \Lambda} \left| e^{\Delta t}(u - \calI_\Lambda(u),\bsx_\bsnu) \right|$. \\
Moreover if $C(\Delta t,\Lambda) ~ < 1$, it holds 
\begin{equation}
\limsup_{n\to \infty} ~ \underset{\bsnu \in \Lambda}{\sup} \left| \mathbb{E} \left[ \tilde{u}^{n}(\bsx_\bsnu) - u(\bsx_\bsnu) \right] \right| \leqslant \dfrac{C_1(\Delta t,\Lambda)}{1 - C(\Delta t,\Lambda)}.
\label{eq:Linftypointwiseconvergence}
\end{equation}
\label{th:th31GM}
\end{theorem}
The condition $C(\Delta t,\Lambda) < 1$ implies that in practice $\Delta t$  should be chosen sufficiently small  \cite[Theorem 4.2]{GM2005}. Under this condition, the error at interpolation points uniformly converges geometrically up 
 to a threshold term depending on time integration errors for interpolation functions $l_\nu$ and the interpolation error $u - \calI_\Lambda (u)$. \\

Theorem \ref{th:th31GM} provides a convergence \blue{result} at interpolation points. 
Below, we provide a corollary to this theorem that provides a convergence result in 
$L^{\blue{\infty}}({\calD})$. This result involves the Lebesgue constants in $L^{\blue{\infty}}$-norm associated to $\calI_\Lambda$, defined by
\begin{equation}
\blue{\calL_{\Lambda}} = \sup_{v \in \blue{\mathcal{C}^0(\overline \calD)}} \dfrac{\|\calI_\Lambda(v)\|_{\blue{\infty}}}{\|v\|_{\blue{\infty}}},
\end{equation}

and such that for any $v \in \blue{\calC^0(\overline \calD)}$, 
\begin{equation}
\| v - \calI_\Lambda(v) \|_{\blue{\infty}} \le ( 1 + \blue{\calL_\Lambda} ) \inf_{w \in V_\Lambda} \| v-w \|_{\blue{\infty}}.
\label{eq:interpolantprojeteorthogonal}
\end{equation}
Throughout this article, we adopt the convention that supremum exclude elements with norm $0$.  
\blue{We recall also that the $L^\infty$ Lebesgue constant can be expressed as $ \calL_{\Lambda}=\sup_{x \in \overline \calD} \sum_{\nu \in \Lambda} |\bsl_\bsnu(\bsx)|$.}
\begin{corollary}[Convergence in $L^{\blue{\infty}}$] ~ \\

Assuming {\it \blue{$(H {2})$}}-{\it $(H {5})$}, it holds
\begin{equation}
\limsup_{n\to \infty}  ~ \| \mathbb{E} \left[ \tilde{u}^{n} - u \right] \|_{\blue{\infty}}
\leqslant \dfrac{C_1(\Delta t,\Lambda)}{1 - C(\Delta t,\Lambda)} \blue{\calL_{\Lambda}}
  + \| u - \calI_\Lambda(u)\|_{\blue{\infty}}.
\label{eq:Linftyconvergence}
\end{equation}
\label{th:Linftyconvergence}
\end{corollary}
\begin{proof}
By triangular inequality, we have
\begin{equation*}
\|\bbE \left[ \tilde u^n - u \right] \|_{\blue{\infty}} \leqslant \| \bbE \left[ \tilde u^n - \calI_\Lambda(u) \right] \|_{\blue{\infty}} + \| \calI_\Lambda(u) - u \|_{\blue{\infty}}.
\end{equation*}
\blue{We can build a continuous function $w$} such that $w(x_\nu) = \bbE \left[ \tilde u^n(x_\nu) - u(x_\nu) \right] $ for all $\nu\in\Lambda$, and such that $$\Vert w \Vert_\infty = \sup_{\nu\in \Lambda} \vert w(x_\nu) \vert =  \sup_{\bsnu \in \Lambda} \left| \mathbb{E} \left[ \tilde{u}^{n}(\bsx_\bsnu) - u(\bsx_\bsnu) \right] \right|.$$
We have then
\begin{align*}
\| \bbE \left[ \tilde u^n - \calI_\Lambda(u) \right] \|_{\blue{\infty}} = \|  \calI_\Lambda(w)  \|_{\blue{\infty}} \le \calL_\Lambda \Vert w \Vert_{\blue{\infty}}.
\end{align*}
The result follows from the definition of the function $w$ and Theorem \ref{th:th31GM}.
\end{proof}
\begin{remark}
\blue{
Since for bounded domains $\calD$, we have
\begin{equation*}
\| v \|_2 \le | \calD |^{1/2} \| v \|_\infty, 
\end{equation*}
for all $v$ in $\calC^0(\overline {\calD})$, where $|\calD|$ denotes the Lebesgue measure of $\calD$, we can deduce the convergence results in $L^2$ norm from those    in $L^\infty$ norm. 
}
\end{remark}

\section{Adaptive sparse interpolation} \label{sec3}

We here present sparse interpolation methods following \cite{Chkifa2013,Chkifa2014}. 

\subsection{ Sparse interpolation} \label{sec:3.1}


For $1\le i \le d $, we let $\{\varphi_{k}^{(i)}\}_{k \in \blue{\bbN_0}}$ be a univariate polynomial basis, where $\varphi^{(i)}_k(x_i)$ is a polynomial of degree $k$. For a multi-index $\nu = (\nu_1,\dots,\nu_d) \in \blue{\mathbb{N}_0^d}$, we introduce the multivariate polynomial 
$$\varphi_\bsnu(\bsx) = \prod_{i=1}^d \varphi_{\nu_i}^{(i)}(x_i).$$
For a subset  $\Lambda \subset \mathbb{N}^d$, we let $\calP_{\Lambda} = \mathrm{span} \{ \varphi_\bsnu : \nu \in \Lambda \}$.
A subset $\Lambda$ is said to be {\em downward closed} if 
$$
\forall \bsnu \in \Lambda, ~ \bsmu \le \bsnu \Rightarrow \bsmu \in \Lambda.
$$
If $\Lambda$ is downward closed, then the polynomial space $\calP_{\Lambda}$ does not depend on the choice of univariate polynomial bases and is such that 
$\calP_{\Lambda} = \mathrm{span} \{ x^\nu : \nu \in \Lambda \}$, with $x^\nu = x_1^{\nu_1} \hdots x_d^{\nu_d}$. \\
In the case where $\calD= \calD_1 \times \ldots \times \calD_d$, we can choose for $\{\varphi^{(i)}_k\}_{k \in \blue{\bbN_0}}$ an orthonormal basis in $L^2(\calD_i)$ (i.e. a rescaled and shifted Legendre basis). Then $\{\varphi_\nu\}_{\nu\in \blue{\mathbb{N}_0^d}}$ is an orthonormal basis of $L^2(\cal D).$ 
To define a set of points $\Gamma_\Lambda$ unisolvent for $\calP_{\Lambda}$, we can proceed as follows. For each dimension $1\le i\le d$, we introduce a sequence of points $\{z^{(i)}_k\}_{k \in \blue{\bbN_0}}$ in $\overline \calD_i$ such that for any $p\ge 0$, $\Gamma^{(i)}_{p} = \{z^{(i)}_k\}_{k= 0}^{p}$ is unisolvent for $\calP_p = \mathrm{span}\{\varphi^{(i)}_k : 0\le k\le p\}$, therefore defining an interpolation operator $\calI^{(i)}_{p}$. Then we let  $$\Gamma_\Lambda = \{ z_\nu = (z_{\nu_1}^{(1)}, \hdots, z_{\nu_d}^{(d)}) : \nu\in \Lambda \} \subset \overline{ \calD}.$$ 
 This construction is interesting for adaptive sparse algorithms since for an increasing sequence of subsets $\Lambda_n$, we obtain an increasing sequence of sets $\Gamma_{\Lambda_n}$, and the computation of the interpolation on $\calP_{\Lambda_n}$ only requires the evaluation of the function on the new set of points $\Gamma_{\Lambda_n} \setminus \Gamma_{\Lambda_{n-1}}.$  Also, with such a construction, we have the following property of the Lebesgue constant of $\calI_{\Lambda}$ in $L^{\blue{\infty}}$-norm. This result is \blue{directly taken from} \cite[Section 3]{Chkifa2014}.
\begin{proposition}
\label{prop:1}
If for each dimension $1\le i \le d$, the sequence of points $\{z_k^{(i)}\}_{k \in \blue{\bbN_0}}$ is such that the interpolation operator $\calI_p^{(i)}$ has a Lebesgue constant
$
\blue{\calL_p} \leq (p+1)^s 
$
for some $s>0$, then for any downward closed set $\Lambda$, the Lebesgue constant $\blue{\calL_\Lambda}$ satisfies
\begin{equation}
\blue{\calL_\Lambda} \leq \left( \# \Lambda \right)^{s+1}.
\end{equation}

\end{proposition}
Leja points or magic points  \cite{MadayMagic} are examples of sequences of points such that  
the interpolation operators $\calI_p^{(i)}$ have Lebesgue constants not growing too fast with $p$.
For a given $\Lambda$ with $\rho_i:=\max_{\nu \in \Lambda} \nu_i$, it is possible to construct univariate interpolation grids 
$\Gamma^{(i)}_{\rho_i}$ with better properties (e.g., Chebychev points), therefore resulting  in better properties for the associated interpolation 
operator $\calI_\Lambda$. However for Chebychev points, e.g., $\rho_i \leq \rho'_i$ does not ensure $\Gamma^{(i)}_{\rho_i} \subset \Gamma^{(i)}_{\rho'_i}$.  Thus with such univariate grids, an increasing sequence of sets $\Lambda_n$ will not be associated with an increasing sequence of sets $\Gamma_{\Lambda_n}$, and the evaluations of the function will not be completely recycled in adaptive algorithms. 
However, for some of the algorithms described in Section \ref{sec:4}, this is not an issue as evaluations can not be recycled anyway. \\
Note that for general domains $\calD$ which are not the product of intervals, the above constructions of grids $\Gamma_\Lambda$ are not viable since it may yield to grids not contained in the domain $\calD$. For such general domains, magic points obtained through greedy algorithms could be considered.


\subsection{Adaptive algorithm for sparse interpolation} \label{sec:32}

An adaptive sparse interpolation algorithm consists in constructing a sequence of approximations $(u_n)_{n\ge 1}$ associated with 
an increasing sequence of downward closed subsets  $(\Lambda_n)_{n \geq 1}$.  
According to \eqref{eq:interpolantprojeteorthogonal}, we have to construct a sequence such that the best approximation error  and the Lebesgue constant are such that
 $$
 \blue{\calL_{\Lambda_n}} \inf_{w \in \calP_{\Lambda_n}} \Vert u - w \Vert_{\blue{\infty}} \longrightarrow 0 \text{ as } n \to \infty
 $$ 
for obtaining a convergent algorithm. For example, if 
\begin{equation}
\inf_{w \in \calP_{\Lambda_n}}\Vert u - w \Vert_{\blue{\infty}} = O((\#\Lambda_n)^{-r})
\end{equation}
holds\footnote{see e.g. \cite{CohenDeVore} for conditions on $u$ ensuring such a behavior of the approximation error.} for some $r > 1$ and if \blue{$\calL_{\Lambda_n} =O(( \#\Lambda_n)^k)$} for $k<r$, then the error $\Vert u - u_n \Vert_{\blue{\infty}} = O(n^{-r'})$ tends to zero with an algebraic rate of convergence $r' = r-k>0$. Of course, the challenge is to propose a practical algorithm that constructs a good sequence of sets $\Lambda_m$. \\
We now present the adaptive sparse interpolation algorithm with bulk chasing procedure introduced in  \cite{Chkifa2013}. Let $\theta$ be a fixed bulk chasing parameter in $(0,1)$ and let $\calE_{\Lambda}(v)= \|P_{\Lambda}(v)\|^2_2$\blue{, where $P_\Lambda$ is the orthogonal projector over $\calP_\Lambda$} for  any subset $\Lambda \subset \blue{\NN_0^d}$.
\begin{algo}\normalfont
\label{algo:unperturbed}
{\bfseries (Adaptive  interpolation algorithm)}
\begin{algorithmic}[1]
\STATE Set $\Lambda_1 = \{\boldsymbol{0}_d\}$ and $n = 1$.
\WHILE{$n \leq N$ and $\varepsilon^{n-1} > \varepsilon$}
\STATE Compute $\calM_{\Lambda_n}$.
\STATE Set $\Lambda_n^\star=\Lambda_n \cup \calM_{\Lambda_n}$ and compute $\calI_{\Lambda_n^\star}(u)$.
\STATE Select $N_n \subset \calM_{\Lambda_n}$ the smallest such that $\calE_{N_n}(\calI_{\Lambda_n^\star}(u)) \ge  \theta \calE_{\calM_{\Lambda_n}}(\calI_{\Lambda_n^\star}(u))$
\STATE Update $\Lambda_{n+1} = \Lambda_n \cup N_n$.
\STATE Compute $u_{n+1}=\calI_{\Lambda_{n+1}}(u)$ (this step is not necessary in practice).
\STATE Compute {$\varepsilon^{n}$}.
\STATE Update  $n = n+1$.
\ENDWHILE
\end{algorithmic}
\end{algo}
At iteration $n$, Algorithm \ref{algo:unperturbed} selects 
a subset of multi-indices $N_n$ in the {\em reduced margin} of $\Lambda_n$ defined by
\begin{equation*}
\calM_{\Lambda_n} = \{ \bsnu \in \bbN^d \setminus \Lambda_n : \forall j \; \text{s.t.} \; \nu_j>0 ,\; \bsnu - \bse_j \in \Lambda_n \},
\end{equation*}
where $(\bse_j)_{k} = \delta_{kj}.$ The reduced margin is such that for any subset $S\subset \calM_{\Lambda_n}$, ${\Lambda_n} \cup S$ is downward closed. 
This ensures that the sequence $(\Lambda_n)_{n\ge 1}$ generated by the algorithm is an increasing sequence 
of downward closed sets.
Finally, Algorithm \ref{algo:unperturbed}  is stopped using a criterion based on   
\begin{equation*}
\varepsilon^n = \dfrac{\calE_{\calM_n}(\calI_{\Lambda_n^\star}(u))}{\calE_{\Lambda_n^\star}(\calI_{\Lambda_n^\star}(u))}. 
\end{equation*}

\section{Combining sparse adaptive interpolation with sequential control variates algorithm} \label{sec:4}

We present in this section two ways of combining Algorithm \ref{algo:GM} and Algorithm \ref{algo:unperturbed}.  First we introduce a perturbed version of Algorithm \ref{algo:unperturbed}  and then an adaptive version of Algorithm \ref{algo:GM}. \blue{At} the end of the section, numerical results will illustrate the behavior of the proposed algorithms.
\vspace{-0.5cm}

\subsection{Perturbed version of Algorithm \ref{algo:unperturbed}}

As we do not have access to exact evaluations of the solution $u$ of  \eqref{eq:ellipticPDE}, Algorithm \ref{algo:unperturbed} can not be used for interpolating $u$. So we introduce a perturbed version of this algorithm,  where the computation of the exact interpolant $\calI_{\Lambda}(u)$ is replaced by an approximation denoted $\tilde u_\Lambda$, which can be computed for example with Algorithm \ref{algo:GM} stopped for a given tolerance $\epsilon_{tol}$ or at step $k$. 
This brings the following algorithm.
\begin{algo}\normalfont
\label{algo:perturbed}
{\bfseries (Perturbed  adaptive sparse interpolation algorithm)}
\begin{algorithmic}[1]
\STATE Set $\Lambda_1 = \{\boldsymbol{0}_d\}$ and $n = 1$.
\WHILE{$n \leq N$ and $\tilde \varepsilon^{n-1} > \varepsilon$}
\STATE Compute $\calM_{\Lambda_n}$.
\STATE Set $\Lambda_n^\star=\Lambda_n \cup \calM_{\Lambda_n}$ and compute $\tilde u_{\Lambda_n^\star}$. \label{steptildeu}
\STATE Select $N_n $ as the smallest subset of $ \calM_{\Lambda_n}$ such that $\calE_{N_n}(\tilde u_{\Lambda_n^\star}) \ge  \theta \calE_{\calM_{\Lambda_n}}(\tilde u_{\Lambda_n^\star})$
\STATE Update $\Lambda_{n+1} = \Lambda_n \cup N_n$.
\STATE Compute $\tilde u_{\Lambda_{n+1}}$. \label{steptildeubis}
\STATE Compute  {$\tilde \varepsilon^{n}$}.
\STATE Update  $n = n+1$.
\ENDWHILE
\end{algorithmic}
\end{algo}

\subsection{Adaptive version of Algorithm \ref{algo:GM}} \label{sec:adaptivealgo1}

As a second algorithm, we consider 
the sequential control variates algorithm (Algorithm \ref{algo:GM}) where at step  \ref{seqvarcomputeerror}, an approximation 
$\tilde e^k$ of $e^k$ is obtained by applying the adaptive interpolation algorithm (Algorithm \ref{algo:perturbed}) to the function $e^{k}_{\Delta t,M}$, which uses Monte-Carlo estimations $e^{k}_{\Delta t,M}(x_\nu)$ of $e^k(x_\nu)$ at interpolation points. At each iteration, $\tilde e^k$ therefore belongs to a different approximation space $\calP_{\Lambda_k}.$  In the numerical section, we will call this algorithm {\bf adaptive Algorithm \ref{algo:GM}}.

\subsection{Numerical results}
In this section, we illustrate the behavior of algorithms previously introduced on different test cases. 
We consider the simple diffusion equation
\begin{equation}
\begin{array}{rcll}
- \triangle u(x) &=& g(x), ~~~ & x \in \calD, \\
u(x) &=& f(x), ~~~ & x \in \partial \calD, \\
\end{array}
\label{eq:testcase}
\end{equation}
were $\calD = ]-1,1[^d$. The source terms and boundary conditions will be specified later for each test case. \\ 
\blue{The stochastic differential equation associated to \eqref{eq:testcase} is the following
\begin{equation}
d X^x_t = \sqrt{2} d W_t, \quad X^x_0 = x, 
\end{equation}
where $(W_t)_{t \ge 0}$ is a $d$-dimensional Brownian motion.} \\
We use tensorized grids of magic points for the selection of interpolation points evolved in adaptive algorithms. \\ 

{\it Small dimensional test case.} We consider a first test case \eqref{TC1} in dimension $d= 5$. Here the source term and the boundary conditions in problem \eqref{eq:testcase} are chosen such that the solution is given by
\begin{equation}
\tag{TC1}
u(x) =x_1^2+sin(x_2)+exp(x_3)+sin(x_4) (x_5+1), \qquad x \in \overline\calD.
\label{TC1}
\end{equation}
\vspace{-0.5cm}
\begin{figure}[H]
\centering 
\includegraphics[scale=0.6]{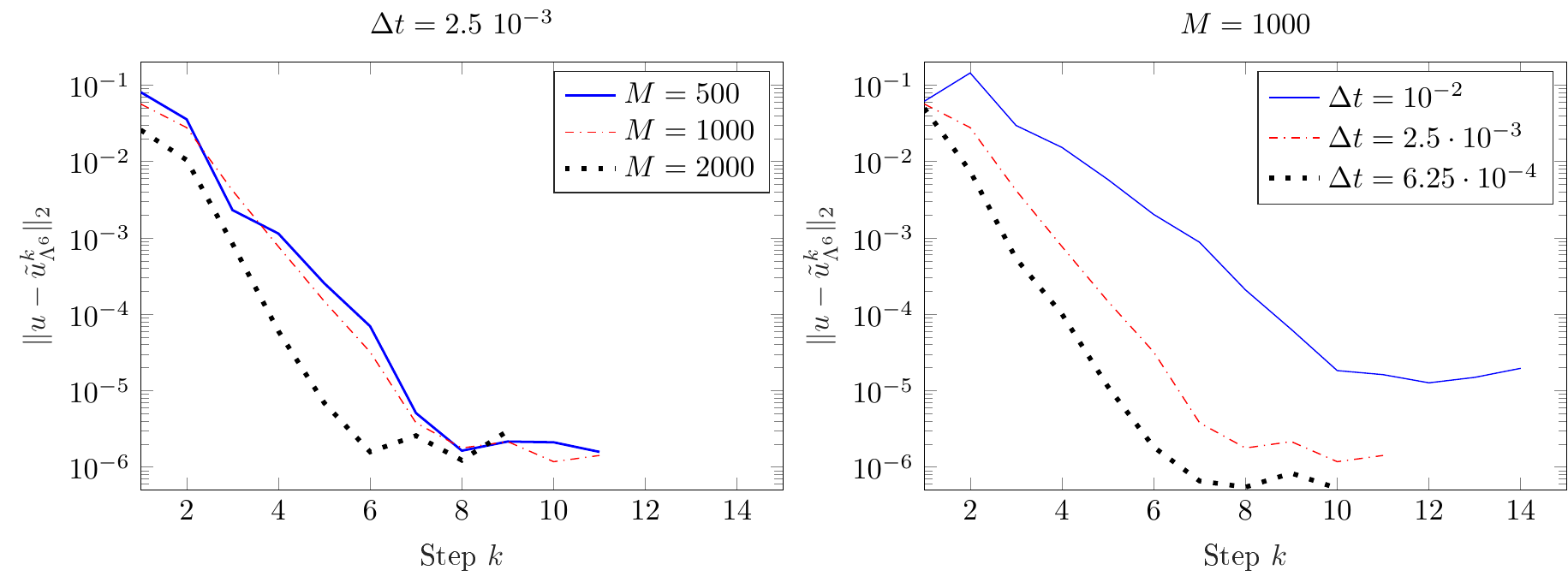}
\caption{(TC1) Algorithm \ref{algo:GM} for fixed $\Lambda$ : evolution of $\|u-\tilde u^k_{\Lambda^6}\|$ with respect to $k$ for \blue{various} $M$ (left figure), and \blue{various} $\Delta t$ (right figure). }
\label{fig:nonadaptiveMvary}
\end{figure}

We first test the influence of $\Delta t$ and $M$ on the convergence of Algorithm \ref{algo:GM} when $\Lambda$ is fixed. 
In that case, $\Lambda$ is selected a priori with Algorithm \ref{algo:unperturbed} using samples of the exact solution $u$ for (TC1), stopped for  $\varepsilon \in \{10^{-6},10^{-8},10^{-10}\}$. In what follows, the notation $\Lambda^i$ stands for the set obtained for $\varepsilon = 10^{-i}, i\in \{6,8,10\}$. 
We represent on Figure \ref{fig:nonadaptiveMvary} the evolution of the absolute error in $L^2$-norm (similar results hold for the $L^\infty$-norm) between the approximation and the true solution with respect to step $k$ for $\Lambda =\Lambda^6$. 
As claimed in Corollary 1, we recover the geometric convergence up to a threshold value that depends on $\Delta t$. We also notice faster convergence as $M$ increases and when $\Delta t$ decreases. We fix $M=1000$ in the next simulations. 
\begin{figure}[H]
\centering
\includegraphics[scale=0.6]{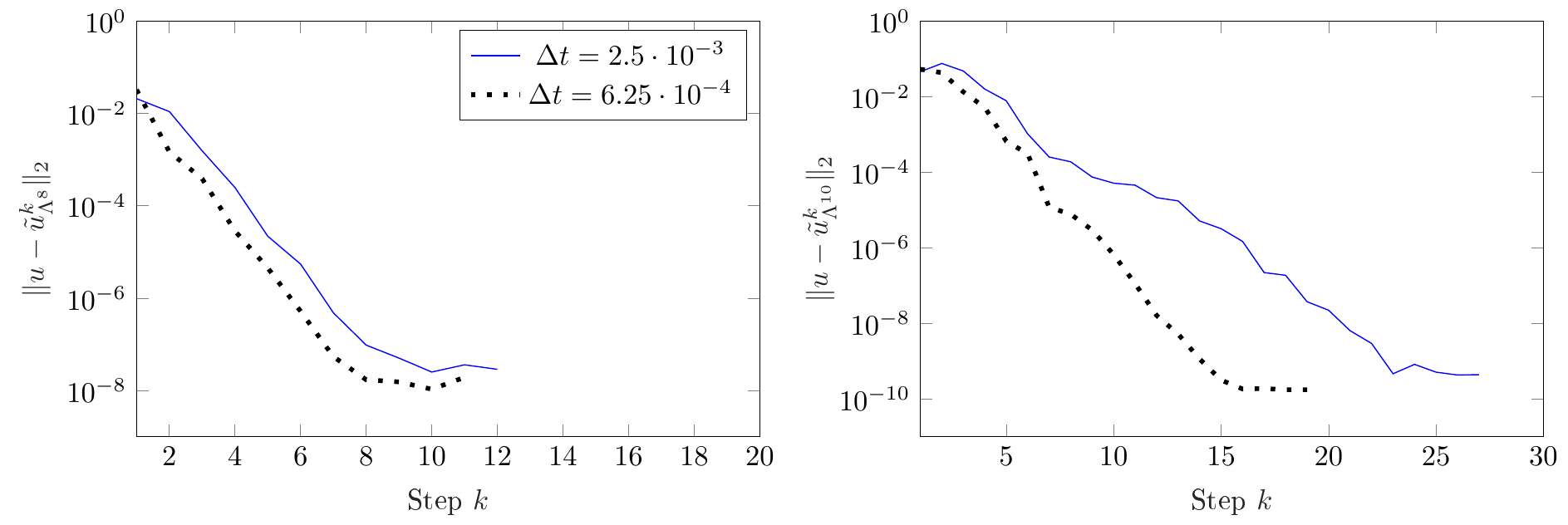}
\caption{(TC1) Algorithm \ref{algo:GM} for fixed $\Lambda^i$
: evolution of $\|u- \tilde u^k_{\Lambda^i}\|_2$ with respect to $k$ for $i=8$ (left figure), and $i=10$ (right figure).}
\label{fig:nonadaptivedeltavary}
\end{figure}
We study the impact of the choice of $\Lambda^i$ on the convergence of Algorithm \ref{algo:GM}.  
Again we observe on Figure \ref{fig:nonadaptivedeltavary} that the convergence rate gets better as  $\Delta t$ decreases. Moreover as $\# \Lambda$ increases the threshold value decreases. This is justified by the fact that interpolation error decreases as $\#\Lambda^i$ increases (see Table \ref{tab:exactalgo2}). Nevertheless, we observe that it may also deteriorate the convergence rate if it is chosen too large together with $\Delta t$ not sufficiently small. Indeed for the same number of iterations $k=10$ and the same time-step $\Delta t = 2.5 \cdot 10^{-3}$, we have an \blue{approximate} absolute error equal to $10^{-\red{7}}$ for $\Lambda^{8}$ against $10^{-4}$ for $\Lambda^{10}$.\\
\begin{table}[H]
\centering
\red{
\begin{tabular}{|c|c|c|c|c|}
\hline
$\Lambda_n$ & \textcolor{black}{$\# \Lambda_n$} & \textcolor{black}{$\varepsilon^n$} & \textcolor{black}{$||u - u_n||_2$} & \textcolor{black}{$||u - u_n||_\infty$} \\
\hline
\hline
~ & 1 &  6.183372e-01 & 1.261601e+00 &  4.213566e+00 \\
~ & 10 &  2.792486e-02 &  1.204421e-01 &  3.602629e-01\\
~ & 20 &  2.178450e-05 &  9.394419e-04 &  3.393999e-03\\
$\Lambda^6$ & 26 &  9.632815e-07 &  4.270457e-06 &  1.585129e-05\\
~ & 30 &  9.699704e-08 &  2.447475e-06 &  8.316435e-06\\
$\Lambda^8$ & 33 &  4.114730e-09 &  2.189518e-08 &  9.880306e-08\\
~ & 40 &  1.936050e-10 &  6.135776e-10 &  1.739848e-09\\
$\Lambda^{10}$ & 41 &  1.008412e-11 &  9.535433e-11 &  4.781375e-10\\
~ & 50 &  1.900248e-14 &  1.004230e-13 &  4.223288e-13\\ 
~ & 55 &  7.453467e-15 &  2.905404e-14 &  1.254552e-13\\
\hline
\end{tabular}
}
\caption{Algorithm \ref{algo:unperturbed} computed on the exact solution of \eqref{TC1}: evolution of $\#\Lambda_n$, error criterion $\varepsilon^n$ and interpolation errors in norms $L^2$ and $L^\infty$ at each step $n$.}
\label{tab:exactalgo2}
\end{table}
We present now the behavior of Algorithm \ref{algo:perturbed}. Simulations are performed \red{with a bulk-chasing parameter} $\theta = 0.5$. At each step $n$ of Algorithm \ref{algo:perturbed}, we use \blue{Algorithm} \ref{algo:GM} with $(\Delta t,M) = (10^{-4},1000)$, \red{stopped when a stagnation is detected}. As shown on the left plot of Figure \ref{fig:perturbed_GM}, for $\#\Lambda_n = 5\red{5}$ we reach approximately a precision of $10^{-14}$ as for \blue{Algorithm} \ref{algo:unperturbed} performed on the exact solution (see Table \ref{tab:exactalgo2}). According to the right plot of  Figure \ref{fig:perturbed_GM}, we also observe that the enrichment procedure behaves similarly  for both algorithms ($\tilde \varepsilon^n$ and $\varepsilon^n$ are almost the same). Here using the approximation provided by Algorithm \ref{algo:GM} has a low impact on the behavior of Algorithm \ref{algo:unperturbed}. \\
\begin{figure}[H]
\centering
\includegraphics[scale=0.6]{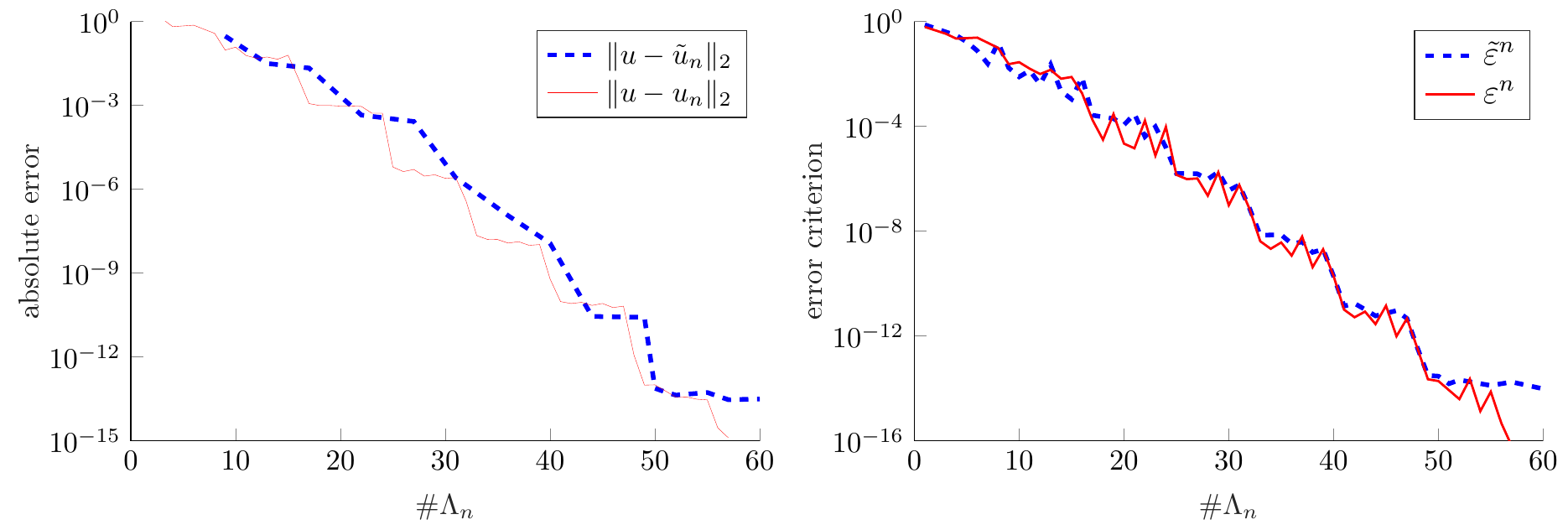}
\caption{(TC1) Comparison of Algorithm \ref{algo:unperturbed} applied to exact solution and Algorithm \ref{algo:perturbed} : (left) absolute error in $L^2$-norm (right) evolution of $\varepsilon^n$ and $\tilde \varepsilon^n$ with respect to $\#\Lambda_n$.}
\label{fig:perturbed_GM}
\end{figure}
We present then results provided with the adaptive Algorithm \ref{algo:GM}. The parameters chosen for the adaptive interpolation are $\varepsilon=5 \cdot 10^{-2}$, $\theta = 0.5$. \red{$K=30$ ensures the stopping of} Algorithm \ref{algo:GM}. As illustrated by Figure \ref{fig:adaptive_GM}, we recover globally the same behavior as for Algorithm \ref{algo:GM} without adaptive interpolation. Indeed as $k$ increases, both absolute errors in $L^2$-norm and $L^\infty$-norm decrease and then stagnate. Again, we notice the influence of $\Delta t$ on the stagnation level. Nevertheless, the convergence rates are deteriorated and the algorithm provides less accurate approximations than Algorithm \ref{algo:perturbed}. This might be due to the sparse adaptive interpolation procedure, which uses here pointwise evaluations based on Monte-Carlo estimates, unlike Algorithm \ref{algo:perturbed} which relies on pointwise evaluations resulting from Algorithm \ref{algo:GM} stopping for a given tolerance.
\begin{figure}[H]
\centering 
\includegraphics[scale=0.6]{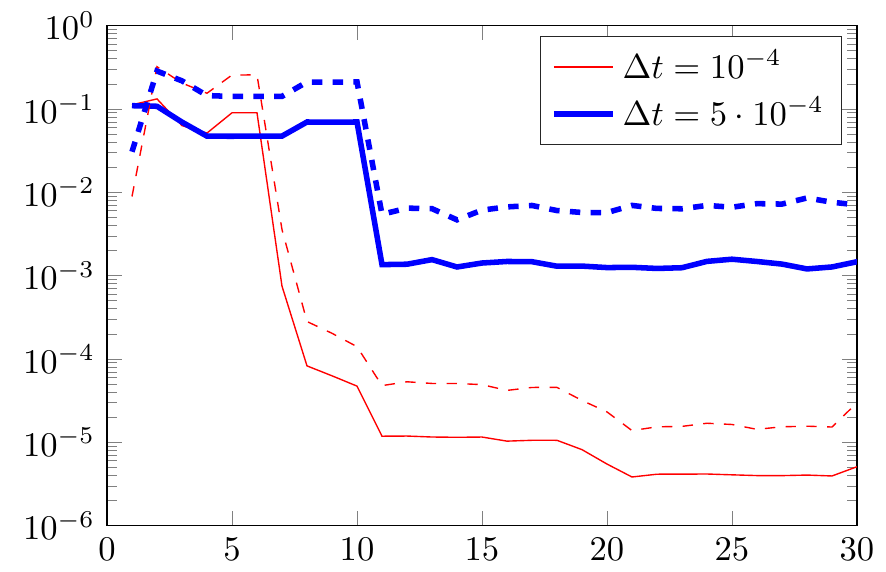}
\caption{(TC1) Adaptive Algorithm \ref{algo:GM}: evolution of  $\|u-u^k_{\Lambda_k}\|_2$ (continuous line) and $\|u-u^k_{\Lambda_k}\|_\infty$ (dashed line) with respect to step $k$ and $\Delta t$.} 
\label{fig:adaptive_GM}
\end{figure}
Finally in Table \ref{tab:compcost}, we compare the algorithmic complexity of these algorithms to reach a precision of $3 \cdot 10^{-5}$ for $(\Delta t,M)=(10^{-4},1000)$. For adaptive Algorithm \ref{algo:GM}, $\Lambda_k$ refers to the set of multi-indices considered at step $k$ of Algorithm \ref{algo:GM}. For Algorithm \ref{algo:perturbed}, $N_n$ stands for the number of iteration required by Algorithm \ref{algo:GM} to reach tolerance $\varepsilon_{tol}$ at step $n$. Finally, Algorithm \ref{algo:GM} is run with full-grid $\Lambda=\Lambda_{max}$ where  $\Lambda_{max} = \{ \nu \in \bbN^d ~:~ \nu_i \leq 10 \}$ is the set of multi-indices allowing to reach the machine precision. In this case, $N$ stands  for the number of steps for this algorithm to converge.
\begin{table}[H]
\centering
\begin{tabular}{|c|c|c|c|c|c|}
\hline
~ & Adaptive Algorithm \ref{algo:GM} & Algorithm \ref{algo:perturbed} & Full-grid Algorithm \ref{algo:GM} \\
\hline
\hline
Th. Complexity & $M  (\Delta t)^{-1} (\sum_k \# \Lambda_k) $ & $M (\Delta t)^{-1} (\sum_n \# \Lambda_n N_n)$ & $M (\Delta t)^{-1}  \# \Lambda_{max} N$ \\
\hline
Est. Complexity & $\red{4} \cdot 10^9$ operations  & $16 \cdot 10^9$ operations & $10^{12} N$ operations \\
\hline
\end{tabular}
\caption{(TC1) Comparison of the algorithmic complexity to reach the precision $3 \cdot 10^{-5}$, with $(\Delta t,M)=(10^{-4},1000)$.}
\label{tab:compcost}
\end{table}
We observe that both the adaptive version of Algorithm \ref{algo:GM} and Algorithm \ref{algo:perturbed} have a similar complexity, which is better than for the full-grid version of Algorithm \ref{algo:GM}. Moreover, we observed that while adaptive version of Algorithm \ref{algo:GM} stagnates at a precision of $3 \cdot 10^{-5}$, Algorithm \ref{algo:perturbed}, with the same parameters $\Delta t$ and $M$, converges almost up to the machine precision. This is why the \blue{high-dimensional} test cases will be run only with Algorithm \ref{algo:perturbed}. \\

{\it Higher-dimensional test cases.} Now, we consider two other test cases noted respectively (TC2) and (TC3) in higher dimension. 
\begin{enumerate}
\item[(TC2)] As second test case in dimension $d=10$,  we define \eqref{eq:testcase} such that its solution is the Henon-Heiles potential $$u(x) = \dfrac{1}{2} \sum_{i=1}^d x_i^2 +0.2 \sum_{i=1}^{d-1} \left( x_i x_{i+1}^2 - x_i^3 \right) + 2.5~10^{-3} \sum_{i=1}^{d-1} \left( x_i^2 + x^2_{i+1} \right)^2, \qquad x \in \overline\calD.$$ 
We set $(\Delta t, M) =(\textcolor{red}{10^{-4}},1000)$ and $K = 30$ for Algorithm \ref{algo:GM}.
\item[(TC3)]
We also consider the problem \eqref{eq:testcase} whose exact solution is a sum of non-polynomial functions, like \eqref{TC1} but now in dimension $d=20$, given by
\begin{equation*}
\begin{split}
u(x) = x_1^2+sin(x_{12})+exp(x_5)+sin(x_{15})(x_8+1).
\end{split}
\end{equation*}
Here, the Monte-Carlo simulations are performed for $(\Delta t, M) = (10^{-4},1000)$ and $K=30$.
\end{enumerate}
Since for both test cases the exact solution is known, we propose to compare the behavior of Algorithm \ref{algo:perturbed}  and  Algorithm \ref{algo:unperturbed}. Again, the approximations $\tilde u_n$, at each step $n$ of Algorithm \ref{algo:perturbed}, are provided by Algorithm \ref{algo:GM} stopped \red{when a stagnation is detected}. 
\textcolor{red}{In both cases, the parameters for Algorithm \ref{algo:perturbed} are set to $\theta = 0.5$ and
$\varepsilon = 10^{-15}$.}\\

In Table \ref{tab:tc2dim10}  and Table \ref{tab:tc3dim20}, we summarize the results associated to the exact and perturbed sparse adaptive algorithms for (TC2) and (TC3) respectively.  We observe that Algorithm \ref{algo:perturbed} performs well in comparison to Algorithm \ref{algo:unperturbed}, for (TC2). Indeed, we get an approximation \red{with a precision below the prescribed value $\varepsilon$ for both algorithms. }
\begin{table}[H]
\centering
\begin{tabular}{|c|c|c|c||c|c|c|c|}
\hline
     $\# \Lambda_n$ & $\varepsilon^n$ & $\| u - u_n\|_\infty $ & $\| u - u_n\|_2 $ & $\#  \Lambda_n$ &  $\tilde \varepsilon^n$  & $\| u - \tilde u_{\Lambda_n}\|_\infty $ & $\| u - \tilde u_{\Lambda_n}\|_2 $ \\
     \hline \hline
      1 & 4.0523e-01 & 3.0151e+00 & 1.2094e+00 & 1 & 3.9118e-01 & 8.3958e-01 & 6.9168e-01 \\
      17 & 1.6243e-01 & 1.8876e+00 &  5.9579e-01 & 17 & 1.6259e-01 & 5.2498e-01 & 3.4420e-01 \\
      36 & 5.4494e-02 & 7.0219e-01 & 2.0016e-01 & 36 & 5.4699e-02 & 1.9209e-01 & 1.2594e-01 \\
      46 & 1.2767e-02 & 1.6715e-01 & 4.9736e-02 & 46 & 1.2806e-02 & 4.6904e-02 & 2.8524e-02 \\
      53 & 9.6987e-04 & 2.9343e-02 & 4.8820e-03 & 53 & 1.0350e-03 & 7.8754e-03 & 2.8960e-03  \\
      60 & 7.6753e-04 & 1.5475e-02 & 4.1979e-03 & 61 & 7.0354e-04 & 3.0365e-03 & 1.7610e-03  \\
      71 & 3.2532e-04 & 8.4575e-03 & 2.1450e-03 & 71 & 3.1998e-04 & 2.3486e-03 & 1.2395e-03  \\
      \textcolor{red}{77} & \textcolor{red}{1.7434e-16} & \textcolor{red}{3.9968e-15} & \textcolor{red}{1.5784e-15} & \textcolor{red}{77} & \textcolor{red}{7.3621e-16} & \textcolor{red}{6.2172e-15} & \textcolor{red}{1.2874e-15}  \\
      \hline
\end{tabular}
\caption{(TC2) Comparison of Algorithm \ref{algo:unperturbed} (first four columns) and Algorithm \ref{algo:perturbed} (last four columns).}
\label{tab:tc2dim10}
\end{table}
\red{Similar observation holds for (TC3) in Table \ref{tab:tc3dim20} and this despite the fact that the test case involves higher dimensional problem.} 
\begin{table}[H]
\centering
\begin{tabular}{|c|c|c|c||c|c|c|c|}
\hline
$\# \Lambda_n$ & $\varepsilon^n$ & $\| u - u_n\|_\infty $ & $\| u - u_n\|_2 $ & $\# \Lambda_n$ &  $\tilde \varepsilon^n$  & $\| u -\tilde u_{\Lambda_n}\|_\infty $ & $\| u -\tilde u_{\Lambda_n}\|_2 $ \\
\hline \hline
1 & 7.0155e-01 & 3.9361e+00 & 1.2194e+00 & 1 & 5.5582e-01 & 7.2832e-01 & 7.0771e-01 \\
6 & 1.4749e-01 & 2.2705e+00 & 5.4886e-01 & 6 & 7.4253e-02 & 2.7579e-01 & 5.1539e-01   \\
11 & 2.1902e-02 & 2.8669e-01 & 1.0829e-01 & 11 & 1.4929e-02 & 4.4614e-02   & 4.1973e-02  \\
15 & 7.6086e-03 & 1.6425e-01 & 4.7394e-02 & 15 & 1.2916e-02 & 1.5567e-02 & 2.5650e-02  \\
20 & 2.2275e-04 & 2.7715e-03 & 7.2230e-04 & 20 & 3.4446e-04 & 5.6927e-04  & 5.3597e-04 \\
24 & 1.4581e-05 & 1.5564e-04 & 7.5314e-05 & 24 & 1.6036e-05 & 2.5952e-05  & 3.0835e-05  \\
30 & 1.8263e-06 & 8.0838e-06 & 2.1924e-06 & 30 & 9.0141e-07 & 2.8808e-06 & 1.9451e-06 \\
35 & 3.9219e-09 & 8.9815e-08 & 2.4651e-08 & 35 & 8.1962e-09 & 2.1927e-08 & 1.5127e-08 \\
40 & 1.7933e-10 & 2.0152e-09 & 6.9097e-10 & 40 & 1.6755e-10 & 2.8455e-10 & 2.6952e-10 \\
45 & 5.0775e-12 & 2.4783e-10 & 4.1600e-11 & 45 & 1.4627e-11 & 3.3188e-11 & 1.7911e-11 \\
49 & 1.7722e-14 & 4.6274e-13 & 8.5980e-14 & 49 & 1.7938e-14 & 8.6362e-14 & 5.0992e-14 \\
54 & 3.9609e-15 & 2.2681e-13 & 3.1952e-14 & 54 & 3.2195e-15 & 4.8142e-14 & 2.6617e-14 \\
56 & 4.5746e-16 & 8.4376e-15 & 3.0438e-15 & 56 & 8.2539e-16 & 8.4376e-15 & 6.3039e-15 \\
 \hline
\end{tabular}
\caption{(TC3) Comparison of Algorithm \ref{algo:unperturbed} (first four columns) and Algorithm \ref{algo:perturbed} (last four columns).}
\label{tab:tc3dim20}
\end{table}

\section{Conclusion}
In this paper we have introduced a probabilistic approach to approximate the solution of high-dimensional elliptic PDEs. This approach relies on adaptive sparse polynomial interpolation using pointwise evaluations of the solution estimated using a Monte-Carlo method with control variates. \\
Especially, we have proposed and compared different algorithms. First we proposed Algorithm \ref{algo:GM} which combines the sequential algorithm proposed in \cite{GM2004} and sparse interpolation. For the non-adaptive version of this algorithm we recover the convergence up to a threshold as  the original sequential algorithm \cite{GM2005}. Nevertheless it remains limited to small-dimensional test cases, since its algorithmic complexity remains high. Hence, for practical use,  the adaptive Algorithm \ref{algo:GM} should be \blue{preferred}. Adaptive Algorithm \ref{algo:GM}  converges but it does not allow to reach low precision with reasonable number of Monte-Carlo samples or time-step in the Euler-Maruyama scheme. 
Secondly, we proposed Algorithm  \ref{algo:perturbed}. It is a perturbed sparse adaptive interpolation algorithm relying on inexact pointwise evaluations of the function to approximate. Numerical experiments have shown that the perturbed algorithm (Algorithm  \ref{algo:perturbed}) performs well in comparison to the ideal one (Algorithm  \ref{algo:unperturbed}) and better than the adapted Algorithm \ref{algo:GM} with a similar algorithmic complexity. 
Here, since only heuristic tools have been provided to justify the convergence of this algorithm, the proof of its convergence, under assumptions on the class of functions to be approximated, should be addressed in a future work. 

\end{document}

%% file: macros.tex
\usepackage{graphicx}        
\usepackage{latexsym}
\usepackage{amsmath}
\usepackage{amsfonts}
\usepackage{amssymb}
\usepackage{url}
\usepackage{algorithm}
\usepackage{algorithmic}

\newcommand{\bse}{{{e}}}

\newcommand{\bsl}{{{l}}}

\newcommand{\bsx}{{{x}}}

\newcommand{\bsX}{{{X}}}

\newcommand{\bsmu}{{{\mu}}}
\newcommand{\bsnu}{{{\nu}}}

\newcommand{\bbE}{{\mathbb{E}}}

\newcommand{\bbN}{{\mathbb{N}}}

\newcommand{\bbP}{{\mathbb{P}}}

\newcommand{\bbR}{{\mathbb{R}}}

\newcommand{\NN}{{\mathbb{N}}} 
\newcommand{\RR}{{\mathbb{R}}} 

\DeclareSymbolFont{bbold}{U}{bbold}{m}{n}
\DeclareSymbolFontAlphabet{\mathbbold}{bbold}

\newcommand{\calA}{{\mathcal{A}}}

\newcommand{\calC}{{\mathcal{C}}}
\newcommand{\calD}{{\mathcal{D}}}
\newcommand{\calE}{{\mathcal{E}}}

\newcommand{\calI}{{\mathcal{I}}}

\newcommand{\calL}{{\mathcal{L}}}
\newcommand{\calM}{{\mathcal{M}}}

\newcommand{\calP}{{\mathcal{P}}}

